\documentclass[art]{amsart}

\usepackage{amssymb,latexsym}
\usepackage{multirow}
\usepackage{amsfonts}
\usepackage{dsfont}
\usepackage[pdftex]{hyperref}
\usepackage{color}
\usepackage{graphicx}
\usepackage[all]{xy}
\usepackage{youngtab,ytableau}

\theoremstyle{plain}
\newtheorem{thm}[equation]{Theorem}
\newtheorem{lem}[equation]{Lemma}

\newtheorem{question}[equation]{Question}
\newtheorem{example}[equation]{Example}
\newtheorem{THM}{Theorem}
\newtheorem{EXAMPLE}[THM]{Example}
\newtheorem{cor}[equation]{Corollary}
\newtheorem{rem}[equation]{Remark}
\newtheorem{defn}[equation]{Definition}
\numberwithin{equation}{section}

\newcommand{\C}{\mathbb C}
\newcommand{\fg}{\mathfrak{g}}

\newcommand{\B}{\mathcal{B}}

\definecolor{grey}{RGB}{180, 187, 198}

\begin{document}

\title[Springer Fibers and Schubert Points]{Springer Fibers and Schubert Points}
\author{Martha Precup}
\address{Dept. of Mathematics and Statistics, Washington University in St. Louis, St. Louis, MO 63130}\email{martha.precup@wustl.edu}

\author{Julianna Tymoczko}
\address{Dept. of Mathematics, Smith College, Northampton, MA 01063}\email{jtymoczko@smith.edu}


%
\maketitle

 \begin{abstract}
 
Springer fibers are subvarieties of the flag variety parametrized by partitions; they are central objects of study in geometric representation theory.  Schubert varieties are subvarieties of the flag variety that induce a well-known basis for the cohomology of the flag variety.  This paper relates these two varieties combinatorially.  We prove that the Betti numbers of the Springer fiber associated to a partition with at most three rows or two columns are equal to the Betti numbers of a specific union of Schubert varieties.

\end{abstract}
  
 
 \section{Introduction}

 This paper proves an explicit combinatorial topological relationship between two families of varieties: certain Springer fibers and certain Schubert varieties.  Both are subvarieties of the flag variety $\B$, whose elements in Lie type $A$ can be described as the collection of nested subspaces $V_{\bullet}=(\{0\}\subseteq V_1\subseteq \cdots \subseteq V_n=V)$ where each $V_i$ is an $i$-dimensional subspace of a fixed complex $n$-dimensional vector space $V$.  The flag variety can also be written as the quotient $\B = GL_n(\mathbb{C})/B$ where $B$ is the subgroup of upper-triangular matrices.
 
Springer fibers are the fibers of a particular desingularization of the nilpotent cone inside the space of $n \times n$ matrices. Explicitly, if $X$ is a nilpotent $n \times n$ matrix then the flag $V_{\bullet}$ is in the Springer fiber $\B^X$ if and only if $X(V_i)\subseteq V_i$ for all $i=1,...,n$.  In other words $\B^X$ consists of all flags that are stable under the operator $X$.  The cohomology of each Springer fiber carries a natural action of the symmetric group that is one of the seminal constructions of geometric representation theory.  Let $\lambda$ be the partition of $n$ determined by the Jordan blocks of $X$.  Springer showed that the top-dimensional cohomology of $\B^X$ is the irreducible representation of $S_n$ corresponding to $\lambda$  \cite{Sp, Sp2} and in fact every irreducible representation of $S_n$ can be obtained this way.

Schubert varieties are subvarieties of the flag variety parametrized by permutations that induce an important basis for the cohomology of the flag variety.  Their geometry is intrinsically connected to the combinatorics of the  symmetric group \cite{F, BL}.   To start, permutations index the double-coset decomposition $\B = \bigsqcup BwB/B$.  Each double coset is an affine cell $C_w = BwB/B$ in the flag variety that contains the permutation flag $wB$ and is called a Schubert cell.  The closure relations between Schubert cells are determined by the Bruhat order on $S_n$ and the dimension of the Schubert variety $\overline{C}_w$ is given by the number of inversions of $w$.  Moreover the cohomology classes of Schubert varieties are given by the Schubert polynomials \cite{BGG}, which are important objects in algebraic combinatorics and representation theory.  The study of Schubert varieties and polynomials fundamentally relates results in geometry, combinatorics, Lie theory, and representation theory; see \cite{BP, CK, Ku} for just a few examples.

The geometry of Springer fibers is much more complicated than Schubert varieties and little is known in general.  Springer fibers are pure dimensional \cite{S} and there are combinatorial formulas for their Betti numbers \cite{Fr, T}.  Because of the challenges of studying $\B^X$ the rest of the literature focuses on special cases, especially when the partition $\lambda$ has two rows, two columns, or is a hook \cite{DH, Fu, Fr2, FM}.  Other work has analyzed the irreducible components of $\B^X$ that are known to be smooth \cite{GZ}.   However even in these special cases, Springer fibers are not fully understood.  For example, there is no general characterization for the closures of the cells paving $\B^X$ even when $\lambda$ has two columns.
  
In this paper we show that if the partition $\lambda$ has at most three rows or two columns then the Betti numbers of $\B^X$ coincide with the Betti numbers of a particular union of Schubert varieties.  We make this correspondence explicit using row-strict tableaux.  The second author showed that the Betti numbers of Springer fibers are enumerated by row-strict Young tableaux of shape $\lambda$ and gave  a combinatorial rule to compute the Betti number corresponding to a given tableau $T$ in \cite{T}.  The number computed from $T$ by this combinatorial rule is called the dimension of $T$.  Using work of Garsia-Procesi and Mbirika \cite{GP, M}, we identify a unique permutation $w_T \in S_n$ associated to each row-strict tableau $T$ with the property that the length $\ell(w_T)$ is the dimension of $T$.  These permutations are called Schubert points.

Our main result is that the Betti numbers of the Springer fibers are the same as the Betti numbers of the union of Schubert varieties corresponding to Schubert points.  

\begin{THM}\label{THM1}  Let $X\in \mathfrak{gl}_n(\C)$ be a nilpotent matrix with Jordan form corresponding to a partition $\lambda$ with at most three rows or two columns.  There is an equality of Poincar\'e polynomials
\[
P(\B^X,t)=P\left( \cup_{wB\in \B^X}\overline{C}_{w_T} , t \right)
\]
where $T$ denotes the row-strict tableau associated to $wB\in \B^X$ and $w_T\in S_n$ is the corresponding Schubert point.   The union simplifies as
\[
	\bigcup_{wB\in \B^X} \overline{C}_{w_T}=\bigcup_{T\in St(\lambda)}\overline{C}_{w_T}
\]
where $St(\lambda)$ denotes the set of standard tableaux of shape $\lambda$.
\end{THM}

This theorem proves a conjecture that arises out of work by Harada and the second author \cite{HT}.  That paper proved that the Betti numbers of the Peterson variety are the same as the Betti numbers of the Schubert variety corresponding to a specific permutation, which suggests a similar result for a larger family of varieties called nilpotent Hessenberg varieties.  This was confirmed for regular nilpotent Hessenberg varieties in Lie type $A$ by Mbirika (who described their Betti numbers \cite{M}) and Reiner (who noted that these Betti numbers agree with the Betti numbers of a kind of Schubert variety called a Ding variety \cite{D, DMR}).  Springer fibers are another special case of nilpotent Hessenberg varieties, so our theorem naturally extends these results. 

The second part of Theorem \ref{THM1} also recovers one of the key conclusions of Springer theory: that the top-dimensional components of $\B^X$ are indexed by standard tableaux, implied both by Springer's original work and geometrically by later work of Spaltenstein \cite{S} and others.

In fact our main result is stronger: it is a bijection on all Betti numbers, not just the top-dimensional ones.  While this paper does not stress the geometric context, this is a dimension-preserving bijection between the Schubert cells in the union $\bigcup_{T\in St(\lambda)}\overline{C}_{w_T}$ and a set of affine cells $C_w \cap \B^X$ that partition the Springer variety.

This bijection is particularly simple for the standard tableaux, or top-dimensional case.  To each standard tableau $T$ of shape $\lambda$ we associate the following Schubert point $w_T$.  If $i$ occurs in the $k^{th}$ row of $T$ set 
\[
w_{i-1}:= \left\{ \begin{tabular}{ll} $s_{i-k+1}\cdots s_{i-2}s_{i-1}$ & if $k>1$\\
				$e$ & if $k=1$ \end{tabular}\right. 
\]
where $s_i$ denotes the simple transposition $(i,i+1)$.  Then $w_T= w_{n-1}w_{n-2} \cdots w_2 w_1$ is the product of these strings.  (Definition \ref{defn: Schubert point} describes Schubert points for general row-strict tableaux.)

\begin{EXAMPLE}\label{intro example}  Consider the partition $\lambda=(2,2,1)$ of $5$.  Below we list all standard tableaux of this shape and the corresponding Schubert points.
\begin{center}
\begin{tabular}{l | c | c | c | c | c}
 \empty & \multirow{3}{*}{$\young(12,34,5)$} & \multirow{3}{*}{$\young(12,35,4)$} & \multirow{3}{*}{$\young(13,24,5)$} & \multirow{3}{*}{$\young(13,25,4)$} & \multirow{3}{*}{$\young(14,25,3)$}\\
 $T\in St(2,2,1)$ & & & & &\\
 & & & & &\\\hline
$w_T\in S_5$ & $s_3s_4 s_3 s_2$ & $s_4 s_2s_3 s_2 $ & $s_3s_4 s_3 s_1$ & $s_4 s_2s_3 s_1$ & $s_4s_1s_2s_1$
\end{tabular}
\end{center}
By Theorem \ref{THM1}, if $X$ has Jordan blocks of dimensions $(2,2,1)$ then the Betti numbers of the Springer fiber $\B^X$ are equal to those of the union $\overline{C}_{s_3s_4 s_3 s_2}\cup \overline{C}_{s_4 s_2s_3 s_2} \cup \overline{C}_{s_3s_4 s_3 s_1}\cup \overline{C}_{s_4 s_2s_3 s_1}\cup \overline{C}_{s_4s_1s_2s_1}$.
\end{EXAMPLE}

Our results also strengthen related work of Garsia-Procesi \cite{GP}.  Garsia and Procesi defined a monomial basis for the cohomology ring of each Springer fiber $\B^X$ that Mbirika later bijectively associated to the row-strict tableaux of shape $\lambda$.  In Mbirika's bijection, the degree of each monomial corresponds to the dimension of the corresponding tableau \cite{M}. (Mbirika's work also provides a readable summary of Garsia-Procesi's algorithm that uses the notation of this paper.)   Schubert points are related to Garsia-Procesi's monomials by the rule
\[
\prod_{i=2}^{n} x_{i}^{\ell_{i-1}}   \longleftrightarrow  w_T : \ell(w_{i-1})=\ell_{i-1}.
\]
For instance, the monomials in the previous example are $x_5^2x_4x_3, x_5x_4^2x_3, x_5^2x_4x_2, x_5x_4^2x_2$, and $x_5x_3^2x_2$.  Garsia and Procesi proved that these monomials are closed under the partial order of division: if $x^{\beta}$ is a monomial corresponding to a row-strict tableau and $x^{\alpha} | x^{\beta}$ then $x^{\alpha}$ corresponds to a row-strict tableau, too.  Similarly, we prove that if $w_T$ is the Schubert point corresponding to a row-strict tableau and $w' \leq w_T$ in Bruhat order then $w'$ is also the Schubert point corresponding to a row-strict tableau.  What makes our result more powerful is that Bruhat order is stronger than the monomial order. For instance $x_3 \not | x_5x_4^2x_2$ but $s_2 < s_4s_2s_3s_1$ in Bruhat order. Indeed, the monomial order is equivalent to erasing some initial simple transpositions from the string $w_i$ while Bruhat order permits erasing simple transpositions in arbitrary locations.  


The methods and results in this paper are combinatorial.  In a second paper we extend these results to parabolic nilpotent Hessenberg varieties  \cite{PT}.  Despite results in these cases and in the regular nilpotent Hessenberg case, we see no straightforward way of extending these methods to general nilpotent Hessenberg varieties or other Springer fibers. The fact that results of this nature hold in so many cases may indicate some deeper geometric phenomenon, such as the degeneration given by Knutson and Miller in \cite{KM} from a Schubert variety to a collection of line bundles.  However, the geometry of  Springer fibers is much less well understood than that of the Schubert varieties, so new methods will be necessary to find such a degeneration.
  
The next section covers background information on the geometry of Springer fibers and the dimensions of row-strict tableaux.  Section \ref{Schubert Points} describes Schubert points and preliminary properties relating them to the permutation flags in Springer fibers.  We prove the main result, Theorem \ref{theorem: main inductive proof}, in Section \ref{section: outline} using a lemma that is proven for three-row tableaux in Section \ref{section: three-row} and for two-column tableaux in Section \ref{section: two-column}.  Section~\ref{section: open questions} poses two open questions related to the constructions herein.

 {\bf Acknowledgements.}  The first author was partially supported by an AWM-NSF mentoring grant.  The second author was partially supported  by National Science Foundation grants DMS-1248171 and DMS-1362855.  The authors are grateful to an anonymous reviewer for insightful questions, including the two open questions appearing in Section~\ref{section: open questions}.


\section{Geometric background on Schubert varieties and Springer fibers}\label{preliminaries}

This section establishes notation and key definitions about Springer fibers.  
 
Let $B$ be the Borel subgroup of $GL_n(\C)$ consisting of upper-triangular matrices. The projective variety $\B=GL_n(\C)/B$ is the flag variety.  As noted in the introduction, the flag variety can be identified with the set of full flags $V_1 \subseteq V_2 \subseteq \cdots \subseteq V_{n-1}\subseteq V$ in a complex $n$-dimensional vector space $V$.  The Weyl group $W$ is the subgroup of permutation matrices in $GL_n(\C)$.  We can identify $W$ with the symmetric group on $n$ letters $S_n$ via the action on column vectors.  The Weyl group is generated by the simple transpositions $s_i = (i,i+1)$.  The Bruhat order on $W$ is defined by the rule that $v \leq w$ if $v$ can be written as a subword of $w$ when each is expressed in terms of the simple transpositions.  If $w$ factors minimally into simple transpositions as $w = s_{i_1} s_{i_2} \cdots s_{i_{\ell(w)}}$ then $\ell(w)$ is the length of $w$.  The length of $w$ is also equal to the number of inversions of $w$.

The Bruhat decomposition partitions the flag variety $\B=\bigsqcup_{w\in S_n} C_w$ into a union of Schubert cells, each of which is induced by a double coset.  The Schubert cell indexed by $w \in S_n$ is the collection of flags $C_w = BwB/B$.  This is in fact a CW-decomposition and it can be proven that $\overline{C}_w=\bigsqcup_{v\leq w} C_v$ where $\leq$ denotes the Bruhat order and $C_v\cong \C^{\ell(v)}$ for all $v\in S_n$.   (See \cite{BL} for a more thorough introduction.)

This description of the Schubert cells allows one to calculate the Poincar\'e polynomial of Schubert varieties using the combinatorics of permutations, as shown in the following example.

\begin{example}\label{schubert cell union}  Let $G=GL_5(\C)$ and consider the union of Schubert varieties from Example \ref{intro example}, $\overline{X}_{s_3s_4 s_3 s_2}\cup \overline{X}_{s_4 s_2s_3 s_2} \cup \overline{X}_{s_3s_4 s_3 s_1}\cup \overline{X}_{s_4 s_2s_3 s_1} \cup \overline{X}_{s_4s_1s_2s_1}$.  The set of all permutations less than or equal to each of $s_3s_4 s_3 s_2$, $s_4 s_2s_3 s_2$, $s_3s_4 s_3 s_1$, $s_4 s_2s_3 s_1$, and $s_4s_1s_2s_1$ respectively in Bruhat order is
\begin{itemize}
\item $s_3s_4 s_3 s_2, s_3s_4s_3, s_3s_4s_2, s_4s_3s_2, s_3s_4, s_4s_3, s_4s_2, s_3s_2, s_4, s_3, s_2, e$,
\item $s_4 s_2s_3 s_2, s_4s_2s_3, s_2s_3s_2, s_4s_3s_2 , s_4s_2, s_4s_3, s_3s_2, s_2s_3, s_4, s_3, s_2, e$,
\item $s_3s_4 s_3 s_1, s_3s_4s_3, s_3s_4s_1, s_4s_3s_1, s_3s_4, s_4s_3, s_3s_1, s_4s_1, s_4, s_3, s_1, e$,
\item $s_4 s_2s_3 s_1, s_4s_2s_3, s_4s_2s_1, s_4s_3s_1, s_2s_3s_1, s_4s_2, s_4s_3,  s_2s_3, s_4s_1, s_3s_1, s_2s_1, s_4, s_3, s_2, s_1, e$, and
\item $s_4s_1s_2s_1, s_4s_1s_2, s_4s_2s_1, s_1s_2s_1, s_4s_2, s_4s_1, s_2s_1, s_1s_2, s_4, s_2, s_1 , e$.
\end{itemize}
Therefore
\[
P(\overline{X}_{s_3s_4 s_3 s_2}\cup \overline{X}_{s_4 s_2s_3 s_2} \cup \overline{X}_{s_3s_4 s_3 s_1}\cup \overline{X}_{s_4 s_2s_3 s_1}\cup \overline{X}_{s_4s_1s_2s_1}, t) = 5t^4+11t^3+9t^2+4t+1.
\]
\end{example}

We now define the subvariety of $\B$ that is the main focus of this manuscript.
 
\begin{defn}[Springer fiber] Let $X$ be an $n\times n$ nilpotent matrix.  The Springer fiber $\B^X$ consists of all flags $gB\in \B$ such that $g^{-1} Xg$ is upper-triangular, or equivalently the flags $V_{\bullet} \in \B$ with $XV_i \subseteq V_i$ for all $i \in \{1,2,\ldots,n\}$.   
\end{defn}

Instead of a CW-decomposition, Springer fibers have a partition called an affine paving.  The closure conditions are weaker in an affine paving than a CW-decomposition but the cells and their dimensions still compute Betti numbers.  (Surveys like Fulton's text have more details \cite{F2}.)  If $X$ is chosen appropriately in its conjugacy class, an affine paving of the Springer fiber $\B^X$ is obtained by intersecting with the Schubert cells.  If $wB$ is a permutation flag in $\B^X$ then we call $w$ a Springer permutation.

The Springer fibers corresponding to $X$ and to any conjugate of $X$ are homeomorphic (this has a one-line proof; see, for example, \cite[Proposition 2.7]{T}) so the Betti numbers of $\B^X$ are an invariant of the conjugacy class of $X$.  When $X$ is nilpotent its conjugacy class is given by the sizes of its Jordan blocks, which we encode as a partition $\lambda$ of $n$.  For this reason we refer to the Betti numbers of $\B^\lambda$ in much of this paper.

We now give a combinatorial description of the Springer permutations and the Betti numbers of $\B^\lambda$.  We start with some basic definitions.

\begin{defn}[Partitions and base fillings] \label{defn: highest form}
Let $\lambda=(\lambda_1, \lambda_2, \ldots, \lambda_k)$ be a partition of $n$ drawn as a Young diagram, namely with $k$ rows of boxes so that the $i^{th}$ row from the top has $\lambda_i$ boxes.  

The base filling of $\lambda$ is obtained as follows.  Fill the boxes of $\lambda$ with integers $1$ to $n$ starting at the bottom of the leftmost column and moving up the column by increments of one.  Then move to the lowest box of the next column and so on.  
\end{defn}

\begin{example}\label{example: highest form} Let $n=5$ and $\lambda = (3,2)$.   The base filling of $\lambda$ is:
	\[
		\young(245,13)
	\]
\end{example}

In fact the row-strict tableaux of shape $\lambda$ parameterize Springer permutations, and a quantity like the inversions of a permutation describe the dimensions of the corresponding affine cell \cite[Theorem 7.1]{T}.

\begin{lem}[Tymoczko] \label{paving lemma}  Fix a partition $\lambda$ of $n$ and consider its base filling.  Suppose that $X$ is the matrix  such that $X_{kj}=1$ if $j$ fills a box directly to the right of $k$ and $X_{kj}=0$ otherwise. The Springer fiber $\B^X$ is paved by affines $C_w\cap \B^X$.  The intersection $C_w\cap \B^X$ is nonempty if and only if $wB\in \B^X$ or equivalently if and only if the filling of $\lambda$ given by labeling the $i^{th}$ box in the  base filling of $\lambda$ by $w^{-1}(i)$ is row-strict.  If $T$ denotes that row-strict tableau of shape $\lambda$, the dimension of $C_w\cap \B^X$ is equal to the number of pairs $(p,q)$ such that  $1\leq p < q \leq n$ and
	\begin{enumerate}
		\item $q$ occurs in a box below $p$ and in the same column or in any column strictly to the left of $p$ in $T$, and
		\item if the box directly to the right of $p$ in $T$ is filled by $r_p$, then $q < r_p$.
	\end{enumerate}
\end{lem}

The dimension formula for the intersection $C_w \cap \B^X$ generalizes the formula for $\ell(w)=\dim(C_w)$ as the inversions of $w$.  To see this, read the numbers in the Young diagram of shape $\lambda$ in the order given by the base filling: the pairs $(p,q)$ described by Condition (1) are precisely the inversions of $w$. These pairs $(p,q)$ are used enough to warrant their own terminology.

\begin{defn}
If $(p,q)$ is a pair with $1\leq p < q \leq n$ that satisfies Conditions (1) and (2) of Lemma \ref{paving lemma} for a row-strict tableau $T$ then we call $(p,q)$ a {\em Springer dimension pair} of $T$. \end{defn}

If $T$ is a row-strict tableau of shape $\lambda$, let $T[i]$ be obtained from $T$ by deleting the boxes labeled by $i+1,...,n$.  Since $T$ is row-strict there are no gaps in the rows of $T[i]$, meaning if a box is deleted then all boxes in the same row and to the right must also be deleted.  Therefore the diagram of $T[i]$ forms a composition of $i$. This gives another way to count Springer dimension pairs.

\begin{lem}\label{countingrows}  Let $\ell_{q-1}$ denote the number of Springer dimension pairs of the form $(p,q)$ where $2\leq q \leq n$.  Then $\ell_{q-1}$ is the sum of
\begin{itemize}
\item the number of rows in $T[q]$ above the row containing $q$ and of the same length, plus 
\item the total number of rows in $T[q]$ of strictly greater length than the row containing $q$.
\end{itemize}
\end{lem}
\begin{proof}  The tableau $T[q]$ has no boxes filled with numbers greater than $q$ so Condition (2) above is satisfied only when $p$ fills a box at the end of a row in $T[q]$.  The rest of the claim follows from imposing Condition (1). 
\end{proof}

\begin{rem}\label{remark: standard tableau}  When $T$ is a standard tableau, the formula above reduces even further.  The entries in both rows and columns are increasing so there are no rows below the row containing $q$ in $T[q]$ of length greater than or equal to the row containing $q$.  (In other words the diagram of $T[q]$ is a partition.)  Therefore $\ell_{q-1}$ simply counts the number of rows above the row containing $q$.
\end{rem}


\section{Schubert Points and combinatorial results about Springer permutations}\label{Schubert Points}
 
We begin by describing a canonical factorization of permutations and some of its properties.  Using this factorization, we define Schubert points, which are permutations corresponding to row-strict fillings of Young diagrams in a different way than Springer permutations. We then give some properties of Schubert points, including many that were observed by Garsia and Procesi and by Mbirika in their earlier studies of essentially the same objects \cite{GP, M}.

Each element of the symmetric group can be factored canonically into monotone-increasing strings of simple reflections, as detailed below \cite[Corollary 2.4.6]{BB}.

\begin{lem}\label{fact: strings} Each $w\in W$ can be written uniquely as $w=w_{n-1}w_{n-2}\cdots w_2w_1$ where 
	\[
		w_i=s_{k_i} s_{k_i+1} \cdots s_{i-1} s_i \textup{   for each  } i=1,...,n-1
	\] 
and either $w_i=e$ or $k_i$ is a fixed integer with $1\leq k_i \leq i$.   Moreover 
\begin{itemize}
\item $\ell(w)=\ell(w_{n-1})+\ell(w_{n-2}) + \cdots + \ell(w_2)+\ell(w_1) $ and
\item \vspace{0.5em} if $w_i \neq e$ then $\ell(w_{i})=i-k_i+1$.
\end{itemize}
The monomial in $\mathbb{Z}[x_1,x_2,\ldots,x_n]$ associated to this factorization is $x_n^{\ell(w_{n-1})} x_{n-1}^{\ell(w_{n-2})} \cdots x_2^{\ell(w_{1})}$.
\end{lem}

We call $w_i$ the $i^{th}$ string of $w$.  For example, the longest word in $S_4$ can be written as $s_1s_2s_3s_1s_2s_1$.  In this case the strings are:
\begin{itemize}
\item $w_3=s_1s_2s_3$
\item $w_2=s_1s_2$
\item $w_1=s_1$
\end{itemize} 
so $k_i=1$ for each $i=1, 2, 3$.
 
 Given a row-strict tableau $T$ we construct the associated Schubert point $w_T$ by using Springer dimension pairs to determine $k_i$ for each $i$.  This produces a permutation $w_T$ whose length is the dimension of the affine cell associated to $T$ in the Springer fiber.  
 
\begin{defn}[Schubert points] \label{defn: Schubert point}  Let $wB\in \B^X$ and let $T$ denote the corresponding row-strict tableau as in Lemma \ref{paving lemma}.  For each $2\leq q\leq n$ let $\ell_{q-1}$ be the number of Springer dimension pairs of the form $(p,q)$ of $T$.  Define a string $w_{q-1}$ by
	\[
	w_{q-1}= \left\{ \begin{tabular}{l l} $s_{q-\ell_{q-1}} s_{q-\ell_{q-1}+1}\cdots s_{q-2}s_{q-1}$ & if $\ell_{q-1}\neq 0$\\
									$e$ & if $\ell_{q-1}=0$ \end{tabular}\right.
	\]
so $w_{q-1}$ is a string of length $\ell_{q-1}$ by construction.  Then
	\[
		w_{T}=w_{n-1} w_{n-2} \cdots w_2 w_1
	\]
is the Schubert point associated to $wB\in \B^X$. We also refer to $w_T$ as one of the Schubert points associated to the partition $\lambda$.
\end{defn}

Our definition together with the properties of the canonical factorization and Lemma \ref{paving lemma} imply that 
\[\ell(w_{T})=\ell_{n-1}+\ell_{n-2}+\cdots + \ell_1=\dim(C_w\cap \B^X).\]

Example \ref{intro example} gave one set of Schubert points.  The next example lists Schubert points corresponding to row-strict fillings other than the standard tableaux.

\begin{example} As in Example \ref{intro example}, let $\lambda = (2,2,1)$.  Below are a few of the row-strict diagrams of this shape and the corresponding Schubert points.  Note that each of the following examples is smaller in Bruhat order than one of the permutations in Example \ref{intro example}.
\begin{center}
\begin{tabular}{l  | c | c | c | c | c | c | c | c}
 \empty & \multirow{3}{*}{$\young(23,14,5)$} & \multirow{3}{*}{$\young(13,45,2)$}  & \multirow{3}{*}{$\young(34,12,5)$} & \multirow{3}{*}{$\young(15,24,3)$} & \multirow{3}{*}{$\young(24,13,5)$} & \multirow{3}{*}{$\young(25,34,1)$} & \multirow{3}{*}{$\young(35,14,2)$} &  \multirow{3}{*}{$\young(35,24,1)$}\\
 $T$ row-strict & & & & & & &\\
 & & & & & & &\\\hline
$w_T\in S_5$ & $s_3s_4s_3$ & $s_4s_3s_1$ & $s_3s_4s_2$ & $ s_1s_2 s_1 $ &  $s_3s_4$ & $s_2$ & $ s_1$ & $e$
\end{tabular}
\end{center}
\end{example}

The association between row-strict tableaux and Schubert points is unique, as Mbirika proved \cite[Section 2]{M} using results of Garsia-Procesi \cite{GP}.

\begin{lem}[Mbirika] \label{Schubert point unique} Given either $wB\in \B^X$ or a row-strict tableau $T$ the corresponding Schubert point $w_T$ is unique. 
\end{lem}

\begin{proof} For each $1<q \leq n$ let $\ell_{q-1}$ be the number of dimension pairs $(p,q)$ of $T$ as in Definition \ref{defn: Schubert point}.  Mbirika showed the map $T \mapsto \prod_{i=2}^{n} x_{i}^{\ell_{i-1}}$ from row-strict tableaux to monomials is an injection that surjects onto a set of monomials defined by Garsia and Procesi \cite[Theorem 2.2.9]{M}.  Each Schubert point $w_T$ is uniquely determined by the numbers $\ell_{q-1}$ for $2\leq q\leq n$ so the claim follows.
\end{proof}

The main theorem in Section \ref{section: outline} proves that the set of Schubert points for various $\lambda$ is closed under the Bruhat order.  We end this section with three results that prove special cases of this main theorem.

The first of these results proves that Schubert points corresponding to standard tableaux are maximal with respect to Bruhat order in the set of all Schubert points for a partition $\lambda$.  (This is independent of the partition $\lambda$.)  

\begin{thm}\label{thm: standard tableaux}  
Let $St(\lambda)$ denote the set of standard tableaux of shape $\lambda$.  Then the Schubert points $\{ w_T: T\in St(\lambda) \}$ are maximal with respect to Bruhat order in the set of all Schubert points for $\lambda$.
\end{thm}

\begin{proof}  Let $T$ be a row-strict tableau with corresponding Schubert point $w_T$.  We will construct a standard tableau $T'\in St(\lambda)$ such that $w_T \leq w_{T'}$.  In fact let $T'$ be the  tableau we obtain from $T$ by reordering the entries in each column so that they increase from top to bottom.  

We first show that the tableau $T'$ is row-strict.  Suppose $r_i$ is the entry in row $i$ and column $k>1$ of $T'$.   Then $r_i$ is greater than $i-1$ other entries of the $k^{th}$ column in $T$.  Since $T$ is row-strict $r_i$ is greater than at least $i$ distinct entries in the $k-1^{st}$ column of $T$.  Thus $r_i$ is greater than the box to its immediate left in $T'$.  So $T'$ is row-strict and by construction also standard.

We claim that $w_T \leq w_{T'}$.  Consider $T[q]$ and $T'[q]$ for $2\leq q \leq n$.  The number of rows of each length is the same in $T[q]$ as in $T'[q]$ because we obtained $T'$ from $T$ by reordering entries within columns.  In particular the rows in $T[q]$ and $T'[q]$ containing $q$ have equal length.  Thus both $T[q]$ and $T'[q]$ have the same number of rows of strictly greater length than the row containing $q$.  Additionally any row in $T[q]$ above the row containing $q$ and of equal length will end in a box in the same column of $T$ as $q$ and be labeled by a value $p<q$.  Since $T'$ reorders the entries of each column of $T$ to increase from top to bottom, this row will also occur above the row containing $q$ in $T'[q]$--- and there may be more rows of this type in $T'[q]$.  Lemma \ref{countingrows} implies that the number of Springer dimension pairs $(p,q)$ in $T$ is at most the number of Springer dimension pairs $(p,q)$ in $T'$.  In other words $\ell(w_{q-1})\leq \ell(w'_{q-1})$ for all $2\leq q \leq n$.  By construction $w_T\leq w_{T'}$ as desired.
\end{proof}

The second claim is a special case of our main theorem, and a slight modification of results of Garsia-Procesi and Mbirika \cite{GP, M}. 

\begin{lem}\label{dissolving case}  Let $T$ be a row-strict tableau of shape $\lambda$ and denote the corresponding Schubert point by $w_T =w_{n-1}w_{n-2}\cdots w_2w_1$.  Suppose that $w'$ is a permutation of the form $w'=w'_{n-1} w'_{n-2} \cdots w'_2 w'_1 $ where $w'_i \leq w_{i}$ in Bruhat order for all $i=1,.., n-1$.  Then $w'$ is a Schubert point associated to $\lambda$. 
\end{lem}
\begin{proof}  We have only to show that there exists a row-strict diagram $T'$ of shape $\lambda$ such that $w_{T'}=w'$.  The monomials associated to $w'$ and $w_T$ according to Lemma \ref{fact: strings} are 
\[
x_{n}^{\ell(w'_{n-1})} x_{n-1}^{\ell(w'_{n-2})} \cdots x_3^{\ell(w'_2)}x_2^{\ell(w'_1)}
 \textup{     and     }
x_{n}^{\ell(w_{n-1})} x_{n-1}^{\ell(w_{n-2})} \cdots x_3^{\ell(w_2)}x_2^{\ell(w_1)} 
\]  
The assumption that $w'_i \leq w_{i}$ implies $\ell(w'_i)\leq \ell(w_i)$ for all $i=1,.., n-1$ so the first monomial divides the second.  Garsia and Procesi proved that it follows that $x_{n}^{\ell(w'_{n-1})} x_{n-1}^{\ell(w'_{n-2})} \cdots x_3^{\ell(w'_2)}x_2^{\ell(w'_1)}$ is an element in their monomial basis for the cohomology of $\B^X$ where $X$ is a nilpotent matrix with Jordan blocks of size $\lambda$ \cite[Proposition 4.2]{GP}.  Let $T'$ denote the row-strict tableau of shape $\lambda$ associated to this monomial by Mbirika \cite[Theorem 2.2.9]{M}.  Lemma \ref{Schubert point unique} thus gives $w_{T'}=w'$. 
\end{proof}

The final result of this section uses dominance order on partitions, which we define below.  

\begin{defn}[Dominance order]
Suppose that $\lambda$ and $\mu$ are two partitions of $n$.  We say $\lambda \geq \mu$
if for each row $i$ we have
\[\lambda_1 + \lambda_2 + \cdots + \lambda_i \geq \mu_1 + \mu_2 + \cdots + \mu_i\]
\end{defn}

Garsia and Procesi showed that divisibility of their monomials respects the dominance order \cite[Proposition 4.1]{GP}, which we restate in our notation below.

\begin{lem} \label{dominance ordering}  Suppose $\lambda, \mu$ are partitions of $n$ with $\lambda \geq \mu$.  If $T$ is a row-strict tableau of shape $\lambda$ associated to Schubert point $w_T$ then there exists a row-strict tableau $T'$ of shape $\mu$ whose Schubert point satisfies $w_{T'}=w_T$. 
\end{lem} 

\begin{proof}  Let $x^{\alpha}$ denote the monomial corresponding to $w_T$.  Garsia and Procesi proved that if $\lambda \geq \mu$ then $x^{\alpha}$ is also a monomial in the basis for the cohomology of $\B^X$ where $X$ is a nilpotent matrix with Jordan blocks of size $\mu$ \cite[Proposition 4.1]{GP}.  Mbirika showed how to construct a row-strict tableau $T'$ of shape $\mu$ with monomial $x^{\alpha}$ \cite[Proof of Theorem 2.2.9]{M}.  Let $w_{T'}$ be the unique Schubert point associated to the tableau $T'$ of shape $\mu$ by Lemma \ref{Schubert point unique}.  Then $w_{T'}=w_{T}$ since both have the same monotone-increasing factorization.
\end{proof}


\section{Outlining the main theorem}\label{section: outline}

In this section we outline and prove the essential lemmas of the main theorem.  The key step in the proof of the main theorem is to carefully follow what happens after one simple reflection is erased from  the monotone-increasing factorization of a Schubert point.  In general erasing one simple reflection produces an extra monotone-increasing string and a factorization that no longer has the form $w_{n-1} w_{n-2} \cdots w_1$   as in Lemma \ref{fact: strings}.  This is the basic situation the following lemma addresses; we keep track of what happens when a monotone-increasing string is conjugated past another.

\begin{lem} \label{whole string lemma}
Let $i$ be a positive integer such that $1\leq i \leq n$ and suppose $1\leq p_i'\leq p_{i} \leq i-1$.  Then
\[
\left(s_{p'_{i}} s_{p'_{i}+1} \cdots s_{p_{i}} \right)  \left(s_{i-\ell_{i-1}} s_{i-\ell_{i-1}+1} \cdots s_{i-1} \right) = \left(s_{i-\ell_{i-1}'} s_{i-\ell_{i-1}'+1} \cdots s_{i-1} \right) \left(s_{p'_{i-1}} s_{p'_{i-1}+1} \cdots s_{p_{i-1}} \right) \]
where $\ell'_{i-1}, p'_{i-1}, p_{i-1}$ are given by the following table:
\[\begin{array}{|c|c|c|c|c|}
\cline{1-5} \hspace{1em} & &  & &  \\
\textup{ Case } & \textup{ Condition } & \ell_{i-1}' = & p_{i-1}' = & p_{i-1}= \\
\cline{1-5} & & \hspace{1em} & & \\
1& p_{i} < i-\ell_{i-1}-1 & \ell_{i-1} & p_{i}' & p_{i} \\
2& p_{i} = i-\ell_{i-1}-1 & \ell_{i-1}+(p_{i}-p_{i}'+1) & N/A & N/A \\
3& p'_{i} \leq i-\ell_{i-1} \leq p_{i} & \ell_{i-1}-1 & p'_{i} & p_{i}-1 \\
 4& i-\ell_{i-1} < p_{i}' & \ell_{i-1} & p'_{i}-1 & p_{i}-1 \\
 \hspace{1em} & &  & &  \\
\hline 
\end{array}\]
\end{lem}

\begin{proof}
If $p_{i} < i-\ell_{i-1}-1$ then each simple reflection in $s_{p'_{i}} s_{p'_{i}+1} \cdots s_{p_{i}}$ commutes with each simple reflection in $s_{i-\ell_{i-1}} s_{i-\ell_{i-1}+1} \cdots s_{i-1}$ which proves the first line of the table.  If $p_{i} = i-\ell_{i-1}-1 $ then the strings glue together to form $s_{p'_{i}} \cdots s_{i-1}$ proving the second line of the table.  If $s_j$ is a simple reflection with $i-\ell_{i-1}<j \leq i-1$ then
\[
s_j  \left(s_{i-\ell_{i-1}} s_{i-\ell_{i-1}+1} \cdots s_{i-1} \right) = \left(s_{i-\ell_{i-1}} s_{i-\ell_{i-1}+1} \cdots s_{i-1} \right) s_{j-1}
\]
using the braid relations.  Repeating this proves the fourth line of the table.  Combining this with the fact that
\[
s_{i-\ell_{i-1}} \left(s_{i-\ell_{i-1}}s_{i-\ell_{i-1}+1} \cdots s_{i-1} \right) = \left( s_{i-\ell_{i-1}+1} \cdots s_{i-1} \right) 
\]
proves the third line of the table.
\end{proof}

We will prove the main theorem by deleting a simple reflection and then rewriting the resulting permutation in monotone-increasing form, one step at a time.  Indeed suppose $T$ is a row-strict tableau of shape $\lambda$ with Schubert point $w_T=w_{n-1}w_{n-2} \cdots w_1$.  When we delete a simple reflection $s_{p_n+1}$ from the initial monotone-increasing string in $w_T$ we obtain
\begin{equation}\label{v-definition}
	 s_{n-\ell_{n-1} } \cdots s_{p_n} \hat{s}_{p_n+1} s_{p_n +2}\cdots s_{n-1} w_{n-2} \cdots w_2 w_1
	= w'_{n-1}  \bigstar_{n-1} \; w_{n-2}\cdots w_2 w_1 
\end{equation}
where $w'_{n-1}=s_{p_n +2}\cdots s_{n-2}s_{n-1} $ and $\bigstar_{n-1}= s_{n-\ell_{n-1} } s_{n-\ell_{n-1} +1} \cdots s_{p_n}$.  On the one hand, if there is a row-strict tableau $T'$ of shape $\lambda$ corresponding to this permutation, it must have $n$ in the box at the end of row 
\[
\ell(w'_{n-1})+1=\ell(s_{p_n+2}\cdots s_{n-2}s_{n-1})+1 = (n-1)-(p_n+2)+2 = n-p_n-1.
\]
Lemma \ref{whole string lemma} then allows us to write $\bigstar_{n-1}w_{n-2}=w'_{n-2}\bigstar_{n-2}$ for (possibly empty) monotone-increasing strings $w'_{n-2}$ and $\bigstar_{n-2}$. The length of $w'_{n-2}$ determines the box in $T'$ where $n-1$ must go, if possible. Continuing this process, the $i^{th}$ step produces the permutation 
\[
	w'_{n-1} w'_{n-2} \cdots w'_{i} \; \bigstar_i \; w_{i-1} \cdots w_2 w_1
\]
where 
\begin{eqnarray*}\label{bigstar definition}
\bigstar_i= s_{p_i '} s_{p_i'+1} \cdots s_{p_i -1}s_{p_i}
\end{eqnarray*} 
for some $p_i, p'_i$ determined by this process.  In the proofs in the next two sections we show that this process terminates and that it results in a row-strict tableau $T'$ of the same shape as $T$, namely that $T'$ is large enough to accommodate each $i$ according to the specifications of $w'_{i-1}$.

We now prove our main result given the following lemma, which will be proven in the next two sections.  

\begin{lem} \label{lemma: inductive step of proof}
Fix a Schubert point $w$ associated to a partition $\lambda$ with at most three rows or two columns.  Suppose that $w'$ is a permutation obtained from $w$ by erasing one simple reflection $s_j$.  Then $w'$ is also a Schubert point associated to the partition $\lambda$.
\end{lem}

The main theorem shows that this lemma implies our claim.

\begin{thm} \label{theorem: main inductive proof}
Suppose that $w$ is a Schubert point associated to a partition $\lambda$ with at most three rows or two columns and $v \leq w$.  Then $v$ is also a Schubert point associated to the partition $\lambda$.
\end{thm}

\begin{proof}
Since $v\leq w$ we can find a string of simple reflections $s_{j_1}, s_{j_2}, \ldots, s_{j_k}$ so that 
\begin{itemize}
\item for each $1\leq i \leq k$ the permutation $v_i$ is obtained from $v_{i-1}$ by erasing one simple reflection $s_{j_i}$ and
\item the initial and terminal permutations are $v_0=w$ and $v_k=v$ respectively.
\end{itemize}
Lemma \ref{lemma: inductive step of proof} says that if $v_{i-1}$ is a Schubert point associated to the partition $\lambda$ then so is $v_i$.  Inducting on $i$ we conclude that $v_k=v$ is a Schubert point associated to $\lambda$ as well.
\end{proof}

\begin{cor} \label{corollary: poincare polynomials} Let $X\in \mathfrak{gl}_n(\C)$ be a nilpotent matrix whose Jordan type is given by the partition $\lambda$ with at most three rows or two columns.  Then the Poincar\'{e} polynomial of the Springer fiber $\B^X$ equals the Poincar\'e polynomial
of the union of Schubert varieties for Schubert points corresponding to standard tableaux of shape $\lambda$:
\[
P(\B^X,t)=P(\cup_{wB\in \B^X} \overline{C}_{w_T},t)= P(\cup_{T\in St(\lambda)} \overline{C}_{w_T},t).
\]
\end{cor}
\begin{proof}  Theorem \ref{theorem: main inductive proof} shows that the set of Schubert points corresponding to $\lambda$ is a lower order ideal with respect to Bruhat order, so the set of these permutations corresponds to a union of Schubert varieties in the flag variety $G/B$.  The second equality follows from Theorem \ref{thm: standard tableaux}.
\end{proof}

\begin{example}  When $\lambda=(2,2,1)$ we have
\[
	P(\B^X,t) = 5t^4+11t^3+9t^2+4t+1
\]
by Corollary \ref{corollary: poincare polynomials} together with Example \ref{schubert cell union}.  The reader can independently verify this fact using the inductive formula for the Poincar\'e polynomial given in \cite{S} or \cite{Fr}.
\end{example}

The following example shows that Lemma \ref{lemma: inductive step of proof} does not hold if $\lambda$ is a partition containing the shape $\mu = (3,1,1,1)$ as a subdiagram.  

\begin{example} 
Let $T$ be following standard tableau of shape $\mu$. 
\begin{center}
$\begin{array}{|c|c|c|}
\hline 1 & 3 & 5\\
\hline 2 \\
\cline{1-1} 4\\
\cline{1-1} 6\\
\cline{1-1}
\end{array}$
\end{center}
$T$ has associated Schubert point $w_T=s_3s_4s_5 s_2s_3s_1$.  Let $w'=s_3\hat{s}_4s_5 s_2s_3s_1 = s_5 s_2s_3s_2s_1$ so by construction $w' \leq w_T$.  However there exists no row-strict filling of $\mu$ corresponding to $w'$!  While Lemma \ref{lemma: inductive step of proof} fails, it is still possible that the Springer fibers have the same Poincar\'{e} polynomials as a union of other Schubert varieties. We have attempted computer calculations to confirm or refute this in the case of $\lambda=(3,1,1,1)$ but so far have not found an algorithm that terminates in reasonable time.  This is the next step in testing whether Theorem~\ref{THM1} generalizes to arbitrary Springer fibers.  All components of $\B^{\lambda}$ are smooth in this case, so one might also consider each of the irreducible components of $\B^\lambda$ separately (see the discussion in Section~\ref{section: open questions}).
\end{example}

Finally the following example demonstrates that these results do not always hold in arbitrary Lie type, not even for partitions with at most two rows.

\begin{example}  Let $\mathfrak{sp}_{6}(\C)$ denote the symplectic Lie algebra of Lie type $C_3$.  The corresponding root system has three simple roots so its Weyl group is generated by three simple reflections.  This means there are precisely three Schubert cells of dimension $1$ in $\B$.  However a well-known result states that if $X\in \fg$ is a subregular nilpotent element, then $\B^X$ is a Dynkin curve \cite[Theorem~6.11]{H}.  For Lie type $C_3$ the Dynkin curve consists of $4$ projective lines (because the associated Dynkin diagram is a path with four edges).  In particular the Poincar\'e polynomial of $\B^X$ is $1+4t$.  There is no Schubert variety or union of Schubert varieties in this flag variety with the same Poincar\'e polynomial.  Therefore our results in this paper do not extend exactly as stated to all Springer fibers in arbitrary Lie type.  
\end{example}


\section{The three row case}\label{section: three-row}

The following theorem proves Lemma \ref{lemma: inductive step of proof} for the three row case.  Recall that if $T$ denotes a row-strict filling of $\lambda$ then $T[i]$ denotes the diagram obtained from $T$ by deleting the boxes labeled by $i+1,...,n$.  We let $\lambda[i]$ denote the partition of $i$ obtained from the composition corresponding to $T[i]$ by reordering the rows in decreasing order.  In this section, we consider the case in which $T$ has at most three rows.  A key feature of this case is that $\ell_{i-1}\leq 2$ by Lemma~\ref{countingrows} since $T[i]$ has at most three rows for all $2\leq i \leq n$, and therefore $\ell(w_{i-1})\leq 2$.  We use this fact in the proof below.

\begin{thm} \label{thm: three row}  Let $\lambda$ be a partition of $n$ with at most three rows and $T$ be a row-strict tableau of shape $\lambda$.  Suppose $w'$ is obtained from $w_T$ by deleting a simple reflection.  Then there exists a row-strict tableau $T'$ of shape $\lambda$ such that $w'=w_{T'}$.
\end{thm}
\begin{proof}  Our proof is by induction on $n$.  We start with the base cases $n \leq 2$.  The cases when $\lambda$ is a single row are trivial, since the Springer fibers in those cases consist of the single flag $eB$ and the only Schubert point is $e$.  The cases when $\lambda$ is a single column are also trivial, since every filling of the diagram is row-strict and hence every permutation flag is in the Springer fiber.  Indeed, the Springer fiber in that case is the full flag variety.  Therefore the Schubert points are also the set of all permutations, namely $\{e, s_1\}$ in the case $n=2$. 

Fix a diagram $\lambda = (\lambda_1, \lambda_2, \lambda_3)$ with $n$ boxes and a row-strict tableau $T$ of shape $\lambda$ with corresponding Schubert point $w_T$.  Consider the following cases:
\begin{itemize}
\item Suppose we do not delete a simple reflection from $w_{n-1}$.  We have $w_T = w_{n-1} w_{T[n-1]}$.  Deleting a simple reflection from $w_T$ results in a  Schubert point for $\lambda$  if and only if deleting a simple reflection from $w_{T[n-1]}$ results in a Schubert point for $\lambda[n-1]$.  The latter holds by induction.
\item Suppose we delete the first simple reflection in the string $w_{n-1}$ reading from the left.  Then the permutation obtained after deleting is $w'_{n-1}w_{n-2} \cdots w_1$ where $w'_{n-1} \leq w_{n-1}$ has either length zero or one.  In all cases $w'_{n-1}w_{n-2} \cdots w_1$ is the Schubert point for a row-strict tableau of shape $\lambda$ by Lemma \ref{dissolving case}.
\end{itemize}

The only case left is when $w_{n-1}=s_{n-2}s_{n-1}$ and we delete $s_{n-1}$.  We now prove that this produces a Schubert point corresponding to a row-strict tableau of shape $\lambda$.  
\begin{itemize}
\item Suppose that $w_{n-2}=e$ so that $w'_{n-2}=s_{n-2}$.  This  means we put $n-1$ into the second row of $\lambda'[n-1]$ and into the first row of $\lambda[n-1]$.  The diagram for $\lambda[n-2]$ thus has (unordered) row-lengths $\{\lambda_1-1, \lambda_2, \lambda_3-1\}$ while the diagram for $\lambda'[n-2]$ has (unordered) row-lengths 
\begin{itemize}
\item[$\bullet$] $\{\lambda_1-1, \lambda_2-1, \lambda_3\}$ if $\lambda_1=\lambda_2=\lambda_3$ or if $\lambda_1 > \lambda_2$, and 
\item[$\bullet$] $\{\lambda_1-2, \lambda_2, \lambda_3\}$ if $\lambda_1 = \lambda_2>\lambda_3$.
\end{itemize}

\begin{figure}
\[\begin{array}{ccc}
\begin{ytableau} \cdots & \empty &  \empty & {\star \bullet} \\
\cdots & \empty & \empty & \empty \\
 \cdots & \empty &  \empty & {\star \bullet}
\end{ytableau} \hspace{0.25in}
 & 
\begin{ytableau} \cdots & \empty & \empty & {\star \bullet} \\
\cdots & \empty & \star \\
 \cdots & \empty &  \bullet \\
\end{ytableau} \hspace{0.25in}
&
\begin{ytableau} \cdots & \empty &  \star & {\star \bullet} \\
\cdots & \empty & \empty & \empty \\
 \cdots & \empty &  \bullet \\
\end{ytableau}
\\
\lambda_1= \lambda_2 = \lambda_3 \hspace{0.25in}& \lambda_1>\lambda_2 \geq \lambda_3 \hspace{0.25in}& \lambda_1=\lambda_2 > \lambda_3 \end{array}
\]
\begin{caption}{Schematics for $\lambda'[n-2]$ and $\lambda[n-2]$. Stars mark boxes erased from $\lambda'[n-2]$ and dots mark boxes erased from $\lambda[n-2]$.} \label{figure: one three-row case}
\end{caption}
\end{figure}
Figure \ref{figure: one three-row case} gives schematics of $\lambda$ in these situations; erasing the boxes with dots gives $\lambda[n-2]$ while erasing the boxes with stars gives $\lambda'[n-2]$.  In all cases the diagrams satisfy 
\[ \lambda'[n-2] \leq \lambda[n-2]\]
and by Lemma \ref{dominance ordering} the product $w_{n-3}\cdots w_1$ corresponds to a  row-strict filling of the shape $\lambda'[n-2]$. Inserting $n-1$ and $n$ at the end of the rows described above produces a row-strict filling $T'$ of shape $\lambda$ such that $w'=w_{T'}$.

\item Suppose that $w_{n-2}=s_{n-2}$ so that $w'_{n-2}=e$.  Then $w'=w_{n-3}w_{n-4} \cdots w_1$ and by Lemma \ref{dissolving case} we know $w'$ is the Schubert point corresponding to some row-strict tableau of shape $\lambda$.

\item Suppose that $w_{n-2}=s_{n-3}s_{n-2}$ so that $w'_{n-2}=w_{n-2}$.  Thus the claim holds if the permutation $w'=w_{n-2} s_{n-3} w_{n-3}w_{n-4} \cdots w_1$ corresponds to a row-strict filling of $\lambda$.  

Suppose that $\lambda_1 \neq \lambda_3$. The row-strict filling $T[n-1]$ of shape $\lambda[n-1]=(\lambda_1, \lambda_2, \lambda_3-1)$ corresponds to the Schubert point $w_{n-2}w_{n-3}\cdots w_1$.  By the inductive hypothesis, there exists a row-strict diagram $T''[n-1]$ of shape $\lambda[n-1]$ corresponding to the permutation 
\[w_{T''[n-1]}=s_{n-3}\hat{s}_{n-2}w_{n-3}w_{n-4}\cdots w_1=s_{n-3}w_{n-3}w_{n-4}\cdots w_1\] 
Moreover $n-1$ occurs at the end of the first row of $T''[n-1]$ since $w_{T''[n-1]}$ does not contain a monotone-increasing string ending in $s_{n-2}$.  Let $T'$ be the row-strict diagram of shape $\lambda$ obtained from $T''[n-1]$ by adding the box corresponding to $n$ back to the third row, filling it with $n-1$, and replacing the label at the end of the first row with $n$.  Then $w_{T'} = w_{n-2} s_{n-3}w_{n-3}w_{n-4}\cdots w_1$ as desired.

Finally if $\lambda_1=\lambda_2 = \lambda_3$ then let $i$ be the largest number so that $w_i=s_{i-1}s_i$ and $w_{i-1}\neq s_{i-2}s_{i-1}$.  Note that 
\[w' = s_{n-2} \left( w_{n-2} w_{n-3} \cdots  w_1 \right) = \left(w_{n-2} w_{n-3} \cdots w_i \right) s_{i-1} \left( w_{i-1} w_{i-2} \cdots w_1\right) \] 
so in particular $w'_j = w_j$ for each $j$ with $i \leq j \leq n-2$.  Also note that 
\[\lambda'[i]=\lambda[i]=(\lambda_1, \lambda_2, \lambda_3-(n-i))\]
since $\lambda_1=\lambda_3$.  If $w_{i-1}=s_{i-1}$ then $w'_{i-1}=e$ while if $w_{i-1}=e$ then $w'_{i-1}=s_{i-1}$.  In both cases
\[\lambda'[i-1]=\lambda[i-1]=(\lambda_1, \lambda_2-1, \lambda_3-(n-i))\]
since $\lambda_1=\lambda_2$.  The permutation $w_{i-2} \cdots w_1$ corresponds to a row-strict filling of $\lambda[i-1]$ by hypothesis.  Since $ w'_{i-2} \cdots w'_1= w_{i-2} \cdots w_1$ and $\lambda'[i-1]=\lambda[i-1]$  we conclude that $w'$ corresponds to a row-strict filling of $\lambda$.
\end{itemize}
\end{proof}


\section{The two column case}\label{section: two-column}

In this final section we prove Lemma \ref{lemma: inductive step of proof} in the case of partitions with two columns. Let $T$ be a row-strict tableau corresponding to such a partition.  Throughout this section we consider each row of $\lambda[i]$ to be labeled with a simple reflection in decreasing order.  Our labeling of the rows of $\lambda[i]$ is inspired and informed by Mbirika's arguments in \cite[Section 2.2]{M}.  We label the second row of $\lambda[i]$ by $s_{i-1}$, the third row by $s_{i-2}$, the fourth row by $s_{i-3}$, and so on, with the top row labeled $e$.  The string $\bigstar_i$ is associated to a subset of boxes, namely those at the end of the rows in $\lambda[i]$ corresponding to the simple reflections in $\bigstar_i$.  We refer to the boxes in these rows as shaded boxes, omitting specific references to $\bigstar_i$.  If a row contains a shaded box, we refer to it as a shaded row.

With this labeling, our algebraic results can be transformed into a claim about boxes in the diagram $\lambda$.  We say that the rightmost box of $\lambda[i]$ in the row labeled by the simple reflection with the lowest index in $w_{i-1}$ is the box corresponding to the string $w_{i-1}$.  Note that there is a bijection between the rows of $\lambda[i]$ and $T[i]$ given as follows.  The row in $\lambda[i]$ labeled by $s_{i-k}$ for some $1\leq k \leq i-1$ corresponds to the row in $T[i]$ having the property that there are $k$ rows above it of greater or equal length, namely the row containing $i$.  The row labeled by $e$ in $\lambda[i]$ corresponds to the top row of length 2 in $T[i]$. This also defines a bijection between between the boxes of $\lambda[i]$ and $T[i]$.  We will use this bijection, and the analogous bijection between the boxes of $\lambda'[i]$ and $T'[i]$, throughout this section.

The simple reflections labeling the rows shift every time $\lambda[i]$ loses a box to become $\lambda[i-1]$.  In particular, the row labeled $s_j$ in $\lambda[i]$ is labeled $s_{j-1}$ in $\lambda[i-1]$. The next table uses this observation to interpret the conditions of Lemma \ref{whole string lemma} graphically.  We say a box touches the shaded boxes if its row is immediately above or below a shaded row.  The table below makes use of the identification between the boxes of $T[i]$ and $\lambda[i]$ established in the paragraph above.

\begin{center}
\begin{tabular}{|l|l|l|}
\cline{1-3}
Case 1 & box containing $i$ is above and  & delete that box from $\lambda[i]$ and \\
& does not touch shaded boxes & slide shaded boxes up one row \\
\cline{1-3} Case 2 & box containing $i$ is above and  & gluing step:   \\
& touches shaded boxes &  un-shade all boxes  \\
\cline{1-3} Case 3 & box containing $i$ is a shaded box & delete that box from $\lambda[i]$ and \\ 
& & un-shade the {\em lowest} shaded box \\
\cline{1-3} Case 4 & box containing $i$ is below & delete box from $\lambda[i]$ \\
&  shaded boxes & \\
\hline
\end{tabular}
\end{center}
We note in particular that in Cases 3 and 4 the shaded boxes do not slide up when passing from $\lambda[i]$ to $\lambda[i-1]$.  The process described in the chart above terminates after all the shaded boxes have been un-shaded.  This occurs after an application of Case 2 (in which the $\bigstar_i$-string glues to $w_{i-1}$) or after several applications of Case 3 (in which case the $\bigstar_i$-string dissolves).

Since a two-column tableau with $n$ boxes is determined by the length of its second column, we write $c_i$ for the length of the second column of $\lambda[i]$.  Note that either $c_{i-1}=c_i$ or $c_{i-1}=c_i-1$ for each $i = 2, 3, \ldots, n$.  With this notation $\ell(w_{i-1})=c_i-1$ is equivalent to saying that $i$ fills the box at the end of the second column of $T[i]$.

We prove below that we can fill the boxes corresponding to $w'_{n-1}, w'_{n-2}, \ldots$ in $\lambda$ with $n, n-1,\ldots$, respectively, up until the gluing step. We create diagrams $\lambda'[n-1], \lambda'[n-2], \ldots$ at each step by erasing the labeled boxes.  The next lemma shows that the partition $\lambda'[i]$ cannot differ much from $\lambda[i]$.

\begin{lem}\label{lemma: cases of c_i}  Let $T$ be a row-strict tableau of shape $\lambda$ with two columns.  It is always possible to fill a row-strict diagram $T'$ according to the length of $w'_{i-1}$ for each $i$ before the gluing step. Moreover $c_i=c_i'$ or $c_i = c_i'+1$ for each $i$ up until the gluing step, according to the following table: 
\end{lem}
\begin{center}
\begin{tabular}{l|c|c}
& $c_i = c_i'$ & $c_i = c_i'+1$ \\
\cline{1-3} Case 1 & $c_{i-1}=c_{i-1}'$ & $c_{i-1} = c_{i-1}'+1$  \\	
 & &  unless $\ell(w_{i-1})=c_i-1$  \\
 & &  in which case $c_{i-1}=c_{i-1}'$ \\
\cline{1-3} Case 3 &  $c_{i-1} = c_{i-1}'$ & $c_{i-1}=c_{i-1}'+1$ \\
& unless $\ell(w'_{i-1})=c_i-1$ & \\
& in which case $c_{i-1}=c_{i-1}'+1$ & \\
\cline{1-3} Case 4  & $c_{i-1}=c_{i-1}'$ & $c_{i-1} = c_{i-1}'+1$ \\	
& & unless $\ell(w_{i-1})=c_i-1$ \\
& & in which case $c_{i-1}=c_{i-1}'$ \\
\end{tabular}
\end{center}

\begin{proof}  We start with the case when $i=n$.  After deleting a simple reflection from $w_{n-1}$ the box containing $n$ in $T'$ must move to a row above the row containing $n$ in $T$.  Either both boxes are in the same column so $c_{n-1}=c_{n-1}'$ or the boxes are in different columns so $c_{n-1}=c_{n-1}'+1$.  

Next we show that this process can be repeated, namely that we can put $i$ in the box of $T'[i]$ that is in bijection with the box corresponding to the string $w_{i-1}'$ in $\lambda'[i]$.  The diagram $\lambda'[i]$ has at least as many rows as $\lambda[i]$ since $c_i'\leq c_i$. In Cases 1 and 4 we know $\ell(w'_{i-1})=\ell(w_{i-1})$ while in Case 3 we know $\ell(w'_{i-1})=\ell(w_{i-1})-1$.  In both cases $\ell(w_{i-1}')\leq \ell(w_{i-1})$ so the diagram $\lambda'[i]$ has a box corresponding to $w_{i-1}'$.  

Finally we show that if $c_i=c_i'$ or $c_i=c_i'+1$ then either $c_{i-1}=c_{i-1}'$ or $c_{i-1} = c_{i-1}'+1$ in Cases~1, 3, and 4. Lemma \ref{whole string lemma} shows that either $i$ goes in the same row in $\lambda'[i]$ as in $\lambda[i]$ or in the row of $\lambda'[i]$ immediately above where $i$ went in $\lambda[i]$.  Consider the diagrams for $c_i=c_i'$ and $c_i=c_i'+1$, with an example of the latter sketched below.
\[
\lambda[i] = \scalebox{0.7}{$\begin{ytableau}  \empty& \empty\\ \empty & \empty\\ \empty & \empty\\\empty \\\empty\\ \end{ytableau}$} \quad \quad \quad
\lambda'[i] = \scalebox{0.7}{$\begin{ytableau}  \empty& \empty\\ \empty & \empty \\ \empty \\ \empty\\ \empty\\ \empty \\ \end{ytableau}$}
\]
The equation relating $c_i$ and $c_i'$ changes only when $i$ fills a box in a different column of $\lambda[i]$ than in $\lambda'[i]$.  There are only two ways this can happen:
\begin{itemize}
\item if $c_i = c_i'+1$ and the box containing $i$ is at the end of the second column of $\lambda[i]$ and just below the second column in $\lambda'[i]$ or
\item if $c_i = c_i'$ and we remove a higher box from $\lambda'[i]$ than from $\lambda[i]$.  
\end{itemize}
The first situation happens in Cases 1 and 4 when $\ell(w_{i-1})=c_i-1$ and the second happens in Case 3 when $\ell(w'_{i-1})=c_i-1$ as claimed.  This resolves all the cases in the table.  
 \end{proof}

\begin{example}  Let $T$ be the following row-strict tableau of shape $\lambda=(2^4, 1^3)$ 
\begin{center} 
\begin{ytableau} 1 & 2\\ 3 & 5\\ 4 & 10\\ 6 & 8\\ 7\\ 11\\ 9 \end{ytableau}
\end{center}
with associated Schubert point
\[
	w_T = s_6s_7 s_8s_9s_{10}\, s_8s_9\, s_4s_5s_6s_7s_8\, s_6s_7\, s_3s_4s_5s_6\, s_3s_4s_5\, s_4 \, s_2s_3\, s_2
\]
and let $w'$ be the permutation obtained by deleting $s_9$ from the initial string $w_{10}=s_6s_7 s_8s_9s_{10}$.  The table below shows the steps we take to simplify $w'$.  Notice that the process ends with the strings gluing at step $i=5$.  At each step the $\bigstar_i$-string is bold.

\begin{center}
\begin{tabular}{ c|c|c|c|c|c }
$i$ & $w'$ & $w_{i-1}$ & $c_i$ & $w_{i-1}'$ & $c_i'$\\ \hline
$11$ & ${\bf s_6s_7 s_8}\hat{s}_9s_{10}\, s_8s_9\, s_4s_5s_6s_7s_8\, s_6s_7\, s_3s_4s_5s_6\, s_3s_4s_5\, s_4 \, s_2s_3\, s_2$ & $s_6s_7 s_8s_9s_{10}$ & $4$ & $s_{10}$ & $4$\\
$10$ & $s_{10}\,{\bf s_6s_7 s_8}\, s_8s_9\, s_4s_5s_6s_7s_8\, s_6s_7\, s_3s_4s_5s_6\, s_3s_4s_5\, s_4 \, s_2s_3\, s_2$ & $s_8s_9$ & $4$ & $s_9$ & $3$\\
$9$ & $s_{10}\, s_9\,{\bf s_6s_7}\, s_4s_5s_6s_7s_8\, s_6s_7\, s_3s_4s_5s_6\, s_3s_4s_5\, s_4 \, s_2s_3\, s_2$ & $s_4s_5s_6s_7s_8$ & $3$ & $s_4s_5s_6s_7s_8$ & $2$\\
$8$& $s_{10}\, s_9\, s_4s_5s_6s_7s_8\,{\bf s_5s_6}\,  s_6s_7\, s_3s_4s_5s_6\, s_3s_4s_5\, s_4 \, s_2s_3\, s_2$ & $s_6s_7$ & $3$ & $s_7$ & $2$\\
$7$ & $s_{10}\, s_9\, s_4s_5s_6s_7s_8\,s_7\, {\bf s_5}\, s_3s_4s_5s_6\, s_3s_4s_5\, s_4 \, s_2s_3\, s_2$ & $s_3s_4s_5s_6$ & $2$ & $s_{3}s_4s_5s_6$ & $1$ \\
$6$ & $s_{10}\, s_9\, s_4s_5s_6s_7s_8\,s_7\, s_3s_4s_5s_6\, {\bf s_4} s_3s_4s_5\, s_4 \, s_2s_3\, s_2$ & $s_3s_4s_5$ & $2$ & $s_3s_4s_5$ & $1$\\
$5$ & $s_{10}\, s_9\, s_4s_5s_6s_7s_8\,s_7\, s_3s_4s_5s_6\, s_3s_4s_5 \,{\bf s_3}\, s_4 \, s_2s_3\, s_2$ & $s_4$ & $1$ & $s_3s_4$ & $1$\\
$4$ & $s_{10}\, s_9\, s_4s_5s_6s_7s_8\,s_7\, s_3s_4s_5s_6\, s_3s_4s_5 \, s_3s_4 \, s_2s_3\, s_2$ & $s_2s_3$ & $1$ & $s_2s_3$ & $1$
\end{tabular}
\end{center}
Case 3 is applied in steps 10 and 8 , Case 4 is applied in step 9, 7, and 6, and Case 2 is applied in step 5 after which the monotone-increasing strings of $w'$ match those of $w$.  Figure \ref{figure: row-strict example} lists the diagrams $T[i]$ and $\lambda[i]$ when $i$ is 10, 9, 8, 7, 6, and 5.  The shaded boxes correspond to the bold strings above.  The box of $w_i$ is bolded for each $i$.

\begin{figure}
\begin{tabular}{c|c|c|c|c|c|c||c}
$i$ & $10$ & $9$ & $8$ & $7$ & $6$ & $5$ & $T'$\\\hline
& & & & & & \\
$T[i]$ &
\begin{ytableau} 1 & 2\\ 3 & 5\\ *(grey)4 & *(grey) {\bf 10}\\ *(grey) 6 & *(grey) 8\\ *(grey) 7\\ \none \\  9 \end{ytableau}&
\begin{ytableau} 1 & 2\\ 3 & 5\\*(grey)  4 \\ *(grey) 6 & *(grey) 8\\  7\\ \none \\ {\bf  9} \end{ytableau} &
\begin{ytableau} 1 & 2\\ 3 & 5\\*(grey)  4 \\ *(grey) 6 & *(grey) {\bf 8}\\  7 \end{ytableau} & 
\begin{ytableau} 1 & 2\\ 3 & 5\\*(grey)  4 \\  6\\  {\bf 7} \end{ytableau} & 
\begin{ytableau} 1 & 2\\ 3 & 5\\*(grey)  4 \\  {\bf 6} \end{ytableau} & 
\begin{ytableau} 1 & 2\\ 3 & {\bf 5}\\*(grey)  4 \end{ytableau} & 
\begin{ytableau} 1 & 2\\ 3 & 11\\ 5& 10\\ 6 & 8\\ 7\\ 9\\ 4 \end{ytableau} \\
& & & & & & \\  \hline
& & & & & & \\
$\lambda[i]$ & 
\begin{ytableau}  \empty & \\  & \\ *(grey) & *(grey) {\bf 10} \\ *(grey)  & *(grey) \\ *(grey)  \\ \\  \end{ytableau}&
\begin{ytableau}  \empty & \\  & \\ *(grey)  & *(grey) \\*(grey) \\  \\ {\bf  9} \end{ytableau} &
\begin{ytableau}  \empty & \\  & \\*(grey)  & *(grey) {\bf 8}\\*(grey)   \\ \\   \end{ytableau} & 
\begin{ytableau}  \empty & \\  & \\*(grey)   \\  \\  {\bf 7} \end{ytableau} & 
\begin{ytableau}  \empty & \\  & \\*(grey)  \\  {\bf 6} \end{ytableau} & 
\begin{ytableau}  \empty & \\  & {\bf 5}\\*(grey) \end{ytableau} & 
\end{tabular}
\begin{caption}{Diagrams for $T[i]$ and $\lambda[i]$.} \label{figure: row-strict example}
\end{caption}
\end{figure}
\end{example}

In the previous example, notice that the first row of length one is shaded in $\lambda[n-1]=\lambda[10]$ and remains shaded in the rest of the steps.  The next lemma proves this holds more generally.

\begin{lem}\label{lemma: top row comes from top row}
Assume that the step from $\lambda[i]$ to $\lambda[i-1]$ is not the gluing step.  The first row of length one is shaded in $\lambda[i-1]$ only if the first row of length one is shaded in $\lambda[i]$.
\end{lem}

\begin{proof}
Assume that the first row of length one is not shaded in $\lambda[i]$.  There are two ways for the top row of length one to become shaded in $\lambda[i-1]$: either
\begin{itemize}
\item the shaded boxes slide up or 
\item the last row of length 2 was shaded, the shaded boxes stay in the same place, and $i$ is placed at the end of the second column of $\lambda[i]$, so $c_{i-1}=c_i-1$.
\end{itemize}

Shaded boxes only slide in Case 1.  In this case, the box containing $i$ is above and separated by at least one row from the shaded boxes.  If the box containing $i$ is in the first column then the shaded boxes are at least two rows below the second column.   Since the first row of length one is not shaded in $\lambda[i]$ either the bottommost shaded box is in a row of length 2 or the topmost shaded box must be at least one row below the top row of length one. The latter case is the only one where the top row of length one could become shaded, and it can only happen if $i$ is in a box in the second column in $\lambda[i]$.  Thus deleting the box containing $i$ from $\lambda[i]$ increases the number of rows of length one in $\lambda[i-1]$.  The shaded boxes slide up, but the additional row of length one above the shaded boxes means the top row of length one is not shaded.

Shaded boxes stay in the same place in Cases 3 and 4.  In Case 4 the box containing $i$ is strictly below the shaded boxes, which cannot occur if this box is in the second column and the last row of length 2 is shaded. So we are reduced to Case 3. In this case, the box containing $i$ is in the second column, so we delete this box leaving an additional shaded row of length one.  The assumption that the top row of length one is not shaded in $\lambda[i]$ implies that the box we just shaded in fact becomes the lowest shaded box.  We un-shade this box, which is precisely the top row of length one in $\lambda[i-1]$.
\end{proof}

The previous lemmas bring us directly to the main theorem.  The crux of the argument is that the only way that gluing could pose a problem is if both $\lambda'[i]=\lambda[i]$ and the shaded boxes contain the top row of length one.  We will show that that cannot happen under our hypotheses.

\begin{thm}
Assume that $\lambda$ is a Young diagram with two columns, that $T$ is a row-strict filling of $\lambda$, and that $w'$ is obtained from $w_T$ by deleting a simple reflection.  There exists a row-strict tableau $T'$ of shape $\lambda$ such that $w'=w_{T'}$.  
\end{thm}

\begin{proof}  As in the proof of Theorem \ref{thm: three row}, we may assume $w'$ is obtained from $w_T$ by deleting a simple reflection from $w_{n-1}$ as in Equation \eqref{v-definition}, else the claim follows by induction on $n$.  If Case 2 (gluing) never applies then $w_{i-1}' \leq w_{i-1}$ for all $i$ by Lemma \ref{whole string lemma} and the claim holds by Lemma \ref{dissolving case}.  So suppose that Case 2 is used at the $i^{th}$ step, namely that the $\bigstar_i$-string glues with $w_{i-1}$.

By Lemma~\ref{lemma: cases of c_i}, we can fill a row-strict diagram $T'$ according to the length of $w_{k-1}'$ for each $k<i$.  At the $i$-th step, note that the lowest shaded row in $\lambda[i]$ is labeled by the simple reflection $s_{p_i'}$ and this is the simple reflection of lowest index in $w_{i-1}' = \bigstar_iw_{i-1}$.   Such a box exists because $\lambda'[i]$ has at least as many rows as $\lambda[i]$ since $c_i'\leq c_i$ by Lemma~\ref{lemma: cases of c_i}.  We then place $i$ in the box of $T'$ that corresponds to the rightmost box in the row labeled by $s_{p_i'}$ in $\lambda'[i]$. This is exactly the box corresponding to the string $w_{i-1}'$ in $\lambda'[i]$.  

The rest of this proof shows that after gluing we have $\lambda[i-1] \geq \lambda'[i-1]$.   The remaining monotone-increasing strings of $w'$ are $w'_{i-2}=w_{i-2}, w'_{i-3}=w_{i-3}, \ldots, w'_1=w_1$.  Lemma \ref{dominance ordering} shows that there is a row-strict filling of shape $\lambda'[i-1]$ corresponding to $w'_{i-2} w'_{i-3} \cdots w'_1$.  This allows us to complete the row-strict filling of $T'[i-1]$ using the bijection between the boxes of $\lambda'[i-1]$ and $T'[i-1]$.  We thus conclude that there is a row-strict tableau $T'$ of shape $\lambda$ corresponding to $w'$.
 
We now prove that $\lambda[i-1] \geq \lambda'[i-1]$ after gluing.  First suppose that the first row of length one in $\lambda[i]$ is not shaded.  In this case the shaded boxes are either all above or all below the top row of length one so gluing means that either
\begin{itemize}
\item the boxes containing $i$ in $\lambda[i]$ and $\lambda'[i]$ are both in the first column, or
\item the boxes containing $i$ in $\lambda[i]$ and $\lambda'[i]$ are both in the second column.
\end{itemize}
We know $c_i'\leq c_i$ or equivalently $\lambda'[i] \leq \lambda[i]$ so deleting a box from the same column in each diagram gives $\lambda'[i-1] \leq \lambda[i-1]$. 

Now suppose that the top row of length one in $\lambda[i]$ is shaded, implying that the box containing $i$ in $\lambda[i]$ is in the second column and the box containing $i$ in $\lambda'[i]$ is in the first column since we are in the gluing case.  If $c_i = c_i'+1$ then after deleting a box from the second column of $\lambda[i]$ and a box from the first column of $\lambda'[i]$ we obtain $c_{i-1}=c_{i-1}'$ so $\lambda[i-1]=\lambda'[i-1]$.  

By Lemma~\ref{lemma: cases of c_i} the only other option is that the top row of length one in $\lambda[i]$ is shaded and that $c_i=c_i'$ or equivalently $\lambda[i]=\lambda'[i]$.  We show that this is in fact impossible.  Indeed, Lemma \ref{lemma: top row comes from top row} says that the top row of length one in $\lambda[i]$ is shaded only if the top row in $\lambda[n-1]$ is shaded.  This means that the box containing $n$ is in the first column of $T$ and the box containing $n$ is in the second column of $T'$ so $c_{n-1}=c_{n-1}'+1$.  Thus there exists a $j$ with $i \leq j \leq n-1$ such that $c_j = c_j'+1$ while $c_{j-1}=c_{j-1}'$.  Inspecting the table in Lemma \ref{lemma: cases of c_i}, we see that this can only happen under very special circumstances.  

By Lemma \ref{lemma: top row comes from top row} the top row of length one in $\lambda[j]$ must be shaded.  The first possibility is that we are in Case 1 and the box containing $j$ is at the bottom of the second column of $\lambda[j]$.  But in Case 1 the box containing $j$ must be above the shaded boxes with at least one row between them, so the first row of length one cannot be shaded.  This contradicts our hypotheses.  The second possibility is that we are in Case 4 and the box containing $j$ is at the bottom of the second column of $\lambda[j]$.  But in Case 4, the shaded rows are above the box containing $j$, and the top row of length one is below this box so the first row of length one cannot be shaded which is again a contradiction.

This proves the case impossible, so after gluing $\lambda[i-1] \geq \lambda'[i-1]$ as desired.
\end{proof}


\section{Open Questions}\label{section: open questions}

We conclude with two open questions and thank an anonymous referee for bringing these to our attention.

\begin{question} Springer fibers are known to have a unimodal distribution of Betti numbers in both the two row and two column cases \cite{FM2}.  Can techniques involving Schubert points be used to give an alternative proof of this fact and extend unimodality to the three row case?
\end{question}
 
The arguments in this manuscript define a correspondence between irreducible components of the Springer fiber and particular Schubert varieties via the bijection between standard tableaux and Schubert points.  In general, the Betti numbers of a given irreducible component of $\B^X$ need not match those of the corresponding Schubert variety.  However, these numbers do agree in the two-row case, in which every irreducible component of the Springer fiber is smooth \cite{FM}.  This motivates the following question.

\begin{question} Do the Betti numbers of a smooth irreducible component of a given Springer fiber agree with those of the corresponding Schubert variety in the two-column and three row cases? 
\end{question}


\end{document}